\documentclass[12pt]{article}
\usepackage{booktabs}
\usepackage{caption}
\usepackage{mathrsfs}
\usepackage{amsmath}
\usepackage{amsfonts,amsthm,amssymb,mathrsfs,bbding}
\usepackage{graphics,multicol}
\usepackage{graphicx}
\usepackage{color}
\usepackage{enumerate}
\usepackage{caption}
\usepackage{rotating}
\usepackage{lscape}
\usepackage{longtable}
\allowdisplaybreaks[4]
\usepackage{tabularx}
\usepackage[colorlinks=true,anchorcolor=blue,filecolor=blue,linkcolor=blue,urlcolor=blue,citecolor=blue]{hyperref}
\usepackage[notref,notcite]{showkeys}
\usepackage{extarrows}
\usepackage{cite}
\usepackage{tikz}
\usepackage{latexsym,bm}
\usepackage{mathtools}
\evensidemargin=\oddsidemargin\topmargin=-1.5cm

\newtheorem{thm}{Theorem}

\newtheorem{lem}[thm]{Lemma}

\newtheorem{claim}{Claim}
\theoremstyle{definition}
\renewcommand\proofname{\it Proof}
\addtocounter{section}{0}
\begin{document}

	\title{\bf Tur\'{a}n problem for $C_{2k+1}^{-}$-free signed graph}
	\author{{Junjie Wang$^{a}$, Yaoping Hou$^{b,}$\footnote{Corresponding author.}\setcounter{footnote}{-1}\footnote{\emph{E-mail address:} yphou@hunnu.edu.cn}, Xueyi Huang$^{a}$}\\[2mm]
		\small $^{a}$School of Mathematics, East China University of Science and Technology,\\
		\small Shanghai 200237, China \\
		\small $^{b}$College of Mathematics and Statistics, Hunan Normal University,\\
		\small Changsha,  Hunan, 410081, P. R. China}

	\date{}
	\maketitle
	{\flushleft\large\bf Abstract } Let $C_{l}^-$ denote a negative cycle of length $l$ in a signed graph. In this paper, we  determine the maximum number of edges among all unbalanced signed graphs of order $n$ with no sugraph switching equivalent to  $C_{2k+1}^-$, where $3\le k\le n/10-1$, and characterize the extremal graphs. As a by-product, we also obtain the maximum spectral radius among these graphs.
	
	\begin{flushleft}
		\textbf{Keywords:} Signed graph; Tur\'{a}n problem; odd cycle.
	\end{flushleft}

	\section{Introduction}
	All graphs considered in this paper are simple and undirected. Let $G$ be a graph with vertex set $V(G)$ and edge set $E(G)$, and let  $e(G)=|E(G)|$. For any $v\in V(G)$, let $N_G(v)$ (or $N(v)$ for short) denote the set of vertices adjacent to $v$ in $G$. For any $e \in E(G)$, if $u$ and $v$ are the two endpoints of $e$, then we can also write $e=\{ uv \}$.  For any two disjoint subsets $S$ and $T$ of $V(G)$, we denote by $E(S,T)$ and $\overline{E}(S,T)$) the set of edges  between $S$ and $T$ in $G$ and in the complement graph of $G$, respectively. Also, let $G[S]$ denote the subgraph of $G$ induced by $S$. The \textit{clique number} of $G$, denoted by $\omega(G)$, is the maximum order of a complete subgraph of $G$. 
	Given two graphs $G_1$ and $G_2$ with $v_1 \in V(G_1)$ and $v_2 \in V(G_2)$, the \textit{coalescence}  of  $G_1$ and $G_2$ with respect to $v_1$ and $v_2$, denote by $G_1 \cdot G_2$, is the graph obtained from $G_1\cup G_2$ by identifying the vertex $v_1$ of $G_1$ with the vertex $v_2$ of $G_2$. 
	
	The \textit{signed graph} $\dot{G}=(G,\sigma)$ is a graph $G$ whose edges get signs $+1$ or $-1,$  where $G$ is called the underlying graph and $\sigma: E(G)\rightarrow  \{+1,-1\}$ is a 
	sign function. The edge $uv$ is \emph{positive} (resp. \emph{negative}) if $uv$ gets  sign $+1$ (resp. $-1$) and denoted by $u\mathop{\sim}\limits^{+} v$ (resp. $u\mathop{\sim}\limits^{-} v$). If all edges  get signs $+1$ (resp. $-1$), then $\dot{G}$ is called \emph{all-positive} (resp. \emph{all-negative}) and  denoted by $(G,+)$  (resp. $(G,-)$). The sign of a cycle $\dot{C}$ is defined by  $\sigma(\dot{C}) = \prod_{e \in E(C)}\sigma (e)$. If $\dot{C}=+1$ (resp. $\dot{C}=-1$), then $\dot{C}$ is called positive (resp. negative). A positive (resp. negative) cycle of length $k$ is denoted by $C_{k}^{+}$ (resp. $C_{k}^{-}$).
	
	The \emph{adjacency matrix} of the signed graph $\dot{G}=(G,\sigma)$ is defined by $A(\dot{G})=(\sigma_{ij}),$ where 
	$\sigma_{ij} =\sigma(v_iv_j)$ if $v_i \sim v_j,$ and $\sigma_{ij} = 0$ otherwise. The eigenvalues of $A(\dot{G})$ are called the \textit{eigenvalues} of $\dot{G}$, and denoted by $\lambda_1(\dot{G}) \ge \lambda_2(\dot{G}) \ge \cdots \ge \lambda_n(\dot{G})$. The \emph{spectral radius} of $\dot{G}$ is defined by $\rho(\dot{G}) =\max \{|\lambda_1(\dot{G})|, |\lambda_2(\dot{G})|, \ldots, |\lambda_n(\dot{G})| \}= \max \{\lambda_1(\dot{G}), -\lambda_n(\dot{G}) \}$.
	
	Since graphs are exactly signed graphs with only positive edges, the properties of graphs can also be considered in terms of signed graphs. On the other hand,  signed graphs can be also used to the structural properties of graphs. For example,  Huang  \cite{H19} completely solved the Sensitivity Conjecture by using the spectral property of signed hypercubes.  For more results about the spectral theory of signed graphs, we refer the reader to \cite{A19,BCKW18,KP,KS17}.

	An important feature of signed graphs is the concept of \emph{switching} the signature. For a subset $U$ of  $V(\dot{G})$,  let $\dot{G}_{U}$ denote the signed graph obtained from $\dot{G}$ by reversing the sign of each edge between $U$ and  $V(\dot{G})\setminus U$. Then we say that $\dot{G}$ and $\dot{G}_U$ are \textit{switching equivalent}, and write $\dot{G} \sim \dot{G}_{U}$. A signed graphs is called \textit{balanced} if it is switching equivalent to a all-positive signed graph. For more basic results on signed graphs, we refer to \cite{Z82,Z08}. 
	
	Let $H$ be a simple graph. A graph $G$ is $H$-\textit{free} if there is no subgraph of $G$ isomorphic to $H$. The \emph{Tur$\acute{a}$n number} $ex(n,H)$ is the maximum number of edges in a $H$-free graph of order $n$. 
	In extremal graph theory, how to determine the  exact value of $ex(n,H)$ is a basic problem, which is also known as the \emph{Tur\'{a}n   problem}. There are many classic results on the \emph{Tur\'{a}n problem}. For example, Mantel (see, e.g., \cite{B78}) gave the exact value of  $ex(n,K_3)$, and  Tur\'{a}n \cite{T54} determined the exact value of $ex(n,K_r)$ for $r\geq 3$. Additionally, the Tur\'{a}n number of odd cycles was determined in \cite{B71,Bo71,WDR,F15}, and the Tur\'{a}n number of even cycles was studied in \cite{B74, F96, F06}. For more results on Tur\'{a}n problem, we refer the readers to the survey \cite{S13}.

	Let $\dot{H}$ be a signed graph. A signed graph $\dot{G}$ is $\dot{H}$-\textit{free} if it has no subgraph switching equivalent to $\dot{H}$. Very recently, Wang, Hou and Li \cite{WHL22} considered the Tur\'{a}n problem for signed graphs.  Let $\dot{G}_{s,t}$  ($s+t=n-2$) be the signed graph obtained from an all-positive clique $(K_{n-2}, +)$ with $V(K_{n-2})=\{u_1, \ldots , u_s, v_1, \ldots, v_t\}$ ($s,t \ge 1$) and two isolated vertices $u$ and $v$ by adding a negative edge $\{uv\}$ and positive edges $\{uu_1 \}, \ldots ,\{uu_s\}, \{vv_1\}, \ldots , \{vv_t\}$ (see Figure \ref{fig1}).  In \cite{WHL22}, 
	the authors determined the maximum number of edges among all $C_3^-$-free connected unbalanced signed graphs of order $n$, and identified the extremal graphs.
	

	\begin{thm}(\cite{WHL22})\label{thm::0}
		Let $\dot{G}=(G,\sigma)$  be a connected unbalanced  signed graph of order $n$.
		If  $\dot{G}$ is  $C_3^-$-free, then  $$e(\dot{G})\le \frac{n(n-1)}{2}-(n-2),$$ with equality holding if and only if $\dot{G}\sim \dot{G}_{s,t},$ where $s+t=n-2$ and $s,t\ge 1$.
	\end{thm} 
	
	Based on Theorem \ref{thm::0}, the authors in \cite{WHL22} also obtained the maximum spectral radius among all $C_3^-$-free connected unbalanced signed graphs of order $n$, and characterized the extremal graphs.
	
	\begin{thm}(\cite{WHL22})\label{thm::00}
		Let $\dot{G}=(G,\sigma)$  be  a connected  unbalanced signed graph of order $n$.
		If  $\dot{G}$ is  $C_3^-$-free, then  $$\rho(\dot{G})\le \frac{1}{2}( \sqrt{ n^2-8}+n-4),$$ with equality holding if and only if $\dot{G}\sim \dot{G}_{1,n-3}.$
	\end{thm}
	
	\begin{figure}
		\begin{center}
			\includegraphics[width=8cm,height=3.5cm]{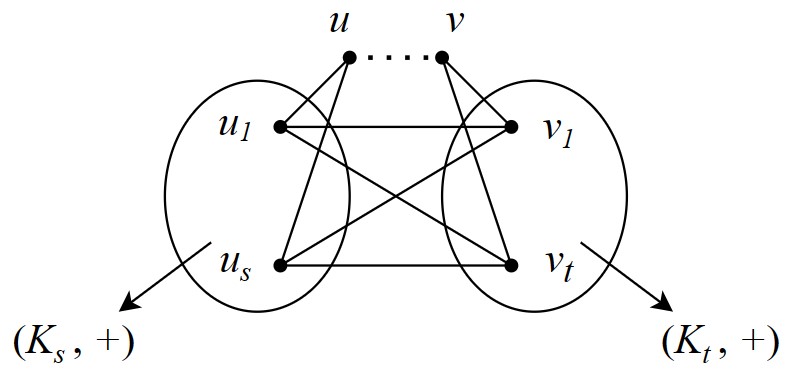}
		\end{center}
		\vskip -0.6cm\caption{The signed graphs $\dot{G}_{s,t}$ (dashed lines indicate negative edges).}
		\label{fig1}
	\end{figure}
	
	In this paper, we consider to extend the results of Theorems \ref{thm::0} and \ref{thm::00}
	to the negative odd cycle $C_{2k+1}^-$ where $k\geq 3$.  Let $C_{3}^{-} \cdot K_{n-2}$ denote the coalescence of $C_{3}^{-}$ and $(K_{n-2},+)$ as shown in Figure \ref{fig2}. Our main results are as follows.

	\begin{figure}
		\begin{center}
			\includegraphics[width=6.5cm,height=2.8cm]{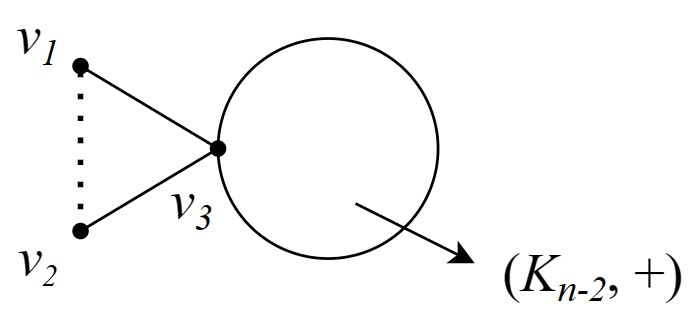}
		\end{center}
		\vskip -0.6cm\caption{The signed graphs $C_{3}^{-} \cdot K_{n-2}$ (dashed lines indicate negative edges).}
		\label{fig2}
	\end{figure}

	\begin{thm}\label{thm::1}
		Let $\dot{G}=(G,\sigma)$ be an unbalanced signed graph with order $n$, and let $3 \le k \le n/10-1$. If $\dot{G}$ is $C_{2k+1}^{-}$-free, then 
		\begin{equation}
			e(\dot{G}) \le \frac{n(n-1)}{2}-2(n-3),
			\nonumber
		\end{equation} 
		with equality holding if and only if $\dot{G} \sim C_{3}^{-} \cdot K_{n-2}$.
	\end{thm}

	\begin{thm}\label{thm::2}
		Let $\dot{G}=(G,\sigma)$ be an unbalanced signed graph with order $n$, and let $3 \le k \le (n-11)/10$.  If $\dot{G}$ is $C_{2k+1}^{-}$-free, then
		\begin{equation}
			\rho(\dot{G}) \le \rho(C_{3}^{-} \cdot K_{n-2}),
			\nonumber
		\end{equation} 
		with equality holding if and only if $\dot{G} \sim C_{3}^{-} \cdot K_{n-2}$. 
	\end{thm}

	\section{Proof of Theorem \ref{thm::1}}
	
	In order to prove Theorem \ref{thm::1}, we first present some key lemmas for later use.

	\begin{lem}(Zaslavsky \cite{Z82})\label{lem::1}
		Let $G$ be a connected graph and $T$ a spanning tree of $G$. The each
		switching equivalence class of signed graphs on the graph $G$ has a unique representative which is +1 on T. Indeed, given any prescribed sign function $\sigma_{T}$: $T \to \{+1,-1\}$, each switching class has a single representative which agrees with $\sigma_{T}$ on $T$.
	\end{lem}
	
	\begin{lem}(Tur\'{a}n \cite{T41})\label{lem::1'}
		Let $G$ be a graph with $n$ vertices, and let  $r$ be a positive integer. If $e(G) > (1-\frac{1}{r})\frac{n^2}{2}$, then $G$ contains $K_{r+1}$ as a subgraph.
	\end{lem}

	For $a\geq 2$, let $H_{n,a}$ denote the graph obtained from $P_{a} \cup K_{n-a}$  by adding an edge between the two endpoints of $P_a$ (unless $a=2$) and connecting these two endpoints to all vertices of  $K_{n-a}$ (see Figure \ref{fig3}).  Suppose that $\dot{H}_{n,a}=(H_{n,a}, \sigma)$ is an unbalanced signed graph with underlying graph $H_{n,a}$.  Let $u$, $v$ be the two endpoints of $P_a$. By Lemma \ref{lem::1}, up to switching equivalence, we can further suppose that  $\sigma(e)=+1$ for all $e \in E(P_a) \cup E(\{u\}, V(K_{n-a}))$. Then one of the following three situations occurs:
	\begin{enumerate}[(i)]
		\item $E(\{v\}, V(K_{n-a}))$ contains only negative edges;
		\item $E(\{v\}, V(K_{n-a}))$ contains only positive edges;
		\item  $E(\{v\}, V(K_{n-a}))$  contains both negative edges and positive edges.
	\end{enumerate}
	For convenience, we use  $\dot{H}_{n,a}^1$, $\dot{H}_{n,a}^2$ and $\dot{H}_{n,a}^3$ to represent $\dot{H}_{n,a}$  in the situations (i), (ii) and (iii), respectively.
	
	If $a=2$, then $\dot{H}_{n,a}=\dot{H}_{n,2}$ is just an unbalanced signed complete graph of order $n$. 
	

	\begin{figure}[t]
		\begin{center}
			\includegraphics[width=5.8cm,height=2.5cm]{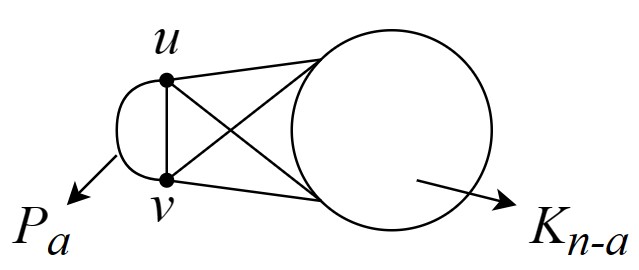}
		\end{center}
		\vskip -0.6cm\caption{The graph $H_{n,a}$.}
		\label{fig3}
	\end{figure}

	\begin{lem}\label{lem::2}
		Let $k$ and $n$ be positive integers with $k\leq (n-3)/2$. Then the unbalanced signed complete graph $\dot{H}_{n,2}$ contains a negative cycle of length $2k+1$.
	\end{lem}
	\begin{proof}
		For $k=1$, by Theorem \ref{thm::0}, $\dot{H}_{n,2}$ contains a negative triangle. For $k \ge 2$, we consider the following three cases. 
		
		{\flushleft {\it Case 1.} $\dot{H}_{n,2}=\dot{H}_{n,2}^1$.}
		
		Take an arbitrary cycle of length $2k-1$ in the subgraph  $\dot{K}_{n-2}$ of $H_{n,2}^1$, and denote it by $\dot{C}_{2k-1}$.  If $\dot{C}_{2k-1}$ is a positive cycle, let  $e_1=\{u_{1}u_{2}\}$ be a positive edge in $\dot{C}_{2k-1}$. Then $\dot{C}_{2k-1}-e_1$ is a positive path, and  $\dot{P}_{2} \cup (\dot{C}_{2k-1}-e_1) \cup \{uu_1\} \cup \{vu_2\}$ is exactly a negative cycle of length $2k+1$ in $\dot{H}_{n,2}^1$.  Similarly, if $\dot{C}_{2k-1}$ is a negative cycle, we can also obtain a positive path $\dot{C}_{2k-1}-e_2$ by removing a negative edge $e_2=\{u_1u_2\}$ from $\dot{C}_{2k-1}$. Then $\dot{P}_{2} \cup (\dot{C}_{2k-1}-e_2) \cup \{uu_1\} \cup \{vu_2\}$ is a negative cycle of length $2k+1$  in $\dot{H}_{n,2}^1$.

		{\flushleft {\it Case 2.} $\dot{H}_{n,2}=\dot{H}_{n,2}^2$.}
		
		Let  $\dot{C}_{2k-1}$ be an arbitrary cycle of length $2k-1$ in the subgraph $\dot{K}_{n-2}$ of $\dot{H}_{n,2}^2$.  If $\dot{C}_{2k-1}$ is a positive cycle that contains a negative edge $e=\{ u_1u_2 \}$, then $\dot{C}_{2k-1}-e$ is a negative path, and  $\dot{P}_{2} \cup (\dot{C}_{2k-1}-e) \cup \{uu_1\} \cup \{vu_2\}$ is a negative cycle of length $2k+1$ in $\dot{H}_{n,2}^2$. If $\dot{C}_{2k-1}$ is a negative cycle that contains a positive edge, as above, we can  find a negative cycle of length $2k+1$ in $\dot{H}_{n,2}^2$. Thus we may assume that, in the subgraph $\dot{K}_{n-2}$ of $\dot{H}_{n,2}^2$, every positive cycle of length $2k-1$  contains only positive edges, and every negative cycle of  length $2k-1$ contains only negative edges. We claim that all edges in $\dot{K}_{n-2}$ are negative. In fact, since $\dot{H}_{n,2}^2$ is unbalanced, there is a negative edge $e_1=\{xy \}$ in $\dot{K}_{n-2}$. If $k=2$, then every triangle in $\dot{K}_{n-2}$ contains only positive edges or only negative edges. This implies that all edges in $E(\{x,y\}, V(K_{n-2})\setminus\{x,y\})$ are negative. For any $e_2=\{z_1z_2\}$ in $\dot{K}_{n-2}-\{x,y\}$, we assert that $e_2$ is a negative edge, since otherwise the vertices $x$, $z_1$ and $z_2$ form a positive triangle with negative edges, contrary to  our assumption. If $k \ge 3$, then for every edge $e_3$ ($e_3\neq e_1$) in $\dot{K}_{n-2}$, there is a cycle of length $2k-1$ in $\dot{K}_{n-2}$ containing both $e_1$ and $e_3$. By assumption, we immediately deduce that $e_3$ is a negative edge. Therefore, we conclude that all edges in the subgraph $\dot{K}_{n-2}$ are negative. 
		Then every cycle of length $2k+1$ (noticing that $2k+1\leq n-2$) in $\dot{K}_{n-2}$ is negative, and hence $\dot{H}_{n,2}^2$ contains a negative cycle of length $2k+2$.
		
		{\flushleft {\it Case 3.} $\dot{H}_{n,2}=\dot{H}_{n,2}^3$.}
		
		According to our assumption, $E(\{v\},V(K_{n-2}))$ contains both positive edges and negative edges. Then there exist two adjacent vertices in the subgraph $\dot{K}_{n-2}$ of $\dot{H}_{n,2}^3$ such that  $\{ vx_1 \}$ is a positive edge and $\{ vx_2 \}$ is a negative edge. Take a cycle $\dot{C}$ of length $2k-1$ in $\dot{K}_{n-2}$ that contains the edge $\{ x_{1}x_{2} \}$. Let $\dot{C}' = \dot{P}_{2} \cup (\dot{C}-\{ x_{1}x_{2} \}) \cup \{ ux_2 \} \cup \{ vx_1 \}$ and $\dot{C}'' = \dot{P}_{2} \cup (\dot{C}-\{ x_{1}x_{2} \}) \cup \{ ux_1 \} \cup \{ vx_2 \} $. Then either $\dot{C}'$ or $\dot{C}''$ is a negative cycle of length $2k+1$, since otherwise $\dot{P}_{2} \cup \{ ux_2 \} \cup \{ vx_1 \}$ and $\dot{P}_{2} \cup \{ ux_1 \} \cup \{ vx_2 \}$ would have the same sign, which is impossible. Therefore, $\dot{H}_{n,2}^3$ contains a negative cycle of length $2k+1$.
		
		This completes the proof.
	\end{proof}

	\begin{lem}\label{lem::2'}
		Let  $k$, $n$ and $a$ be integers with $a\geq 3$ and $(a-1)/2 \le k \le (n-a-1)/2$. Then  $\dot{H}_{n,a}$ contains a negative cycle of length $2k+1$. 
	\end{lem}
	\begin{proof}
		Observe that the underlying graph $H_{n,a}$ contains $K_{n-a+2}$ as a subgraph. If $\dot{K}_{n-a+2}$ is unbalanced, by Lemma \ref{lem::2},  $\dot{K}_{n-a+2}$ contains a negative cycle of length $2k+1$ because $k \le (n-a-1)/2$. Thus $\dot{H}_{n,a}$ contains a negative cycle of length $2k+1$. Now suppose that $\dot{K}_{n-a+2}$ is balanced. Up to switching equivalence, we may assume that $\sigma(e)=+1$ for all $e \in E(K_{n-a+2})$. Then $\dot{P}_a$ is a negative path because $\dot{H}_{n,a}$ is unbalanced. Let $\dot{P}$ be a path of length $2k-a+2$ in $\dot{K}_{n-a+2}$ with endpoints $u$ and $v$. Then  $\dot{P}_a \cup \dot{P}$ forms a negative cycle of length $2k+1$. Therefore, $\dot{H}_{n,a}$ contains a negative cycle of length $2k+1$.
	\end{proof}

	\begin{lem}\label{lem::3}
		Let $\dot{G}=(G,\sigma)$ be an unbalanced signed graph with order $n$. If $e(\dot{G}) \ge \frac{n(n-1)}{2}-2(n-3)$, then $\dot{G}$ contains a negative cycle of length at most $4$.
	\end{lem}
	\begin{proof}
		Let $C^{-}$ be a negative cycle of minimum length in $\dot{G}$, and let $c$ denote the length of $C^{-}$. By contradiction, suppose that $c\geq 5$. We consider the following two situations. 
		
		{\flushleft {\it Case 1.} $c=5$.}	
		
		Let $V(C^{-})=\{ v_1, v_2, v_3, v_4, v_5\}$. Up to switching equivalence, we let $\sigma(\{v_4v_5\})=-1$ and $\sigma(\{v_1v_2\})=\sigma(\{v_2v_3\})=\sigma(\{v_3v_4\})=\sigma(\{v_5v_1\})=1$. First we claim that $|N_{C^{-}}(x)| \le 3$ for any $x \in V(G)\setminus V(C^{-})$.  In fact, if $x$ is adjacent to the vertices $v_1$, $v_3$, $v_4$ and $v_5$, then $\{xv_4 \}$ and $\{xv_5 \}$ must have different signs, since otherwise  $\{xv_4 \} \cup \{xv_5 \} \cup \{v_4v_5 \}$ would form a negative triangle, contrary to our assumption. Without lose of generality, let $\sigma(xv_4)=-1$ and $\sigma(xv_5)=1$. Consider the cycles $v_1 \rightarrow v_5 \rightarrow x \rightarrow v_1$ and  $v_3 \rightarrow v_4 \rightarrow x \rightarrow v_3$. Since both of them are of length less than $5$, by our assumption, we can deduce that $\sigma(v_1x)=1$ and $\sigma(v_3x)=-1$. Then the cycle $v_1 \rightarrow v_2 \rightarrow v_3 \rightarrow x \rightarrow v_1$ is a negative cycle of length less than $5$, a contradiction. For the remaining cases, by using a similar method, we can also obtain a negative cycle of length less than $5$,  which is impossible.  Furthermore, we assert that $C^{-}$ has no chords, since otherwise the chord would divide $C^-$ into two cycles of length less than $5$, and one of them must be negative. Therefore, 
		$$
		e(\dot{G})\leq \frac{n(n-1)}{2}-2(n-5)-5=\frac{n(n-1)}{2}-(2n-5)<\frac{n(n-1)}{2}-2(n-3),
		$$
		contrary to the assumption. 
		
		{\flushleft {\it Case 2.} $c \ge 6$.}
		
		Let $x$ and $x'$ be a pair of vertices at distance $\lfloor c/2 \rfloor$ in the cycle $C^-$.  First we claim that  $N_{G-C^{-}}(x) \cap N_{G-C^{-}}(x') = \emptyset$. In fact, if there exists some $y \in N_{G-C^{-}}(x)$ $\cap$ $N_{G-C^{-}}(x')$, then the path $x \rightarrow y \rightarrow x'$ divides the cycle $C^{-}$ into two cycles of length less than  $c$, and one of them must be negative, contrary to our assumption. Thus we have $e(\{x,x'\},V(G-C^-))\leq n-c$. Furthermore, we assert that $C^-$ has no chords, since otherwise the chord would divide $C^-$ into two cycles of length less than $c$, and one of them must be negative. Let $m_1$ and $m_2$ denote the number of edges missing in $E(V(C^-),V(G)\setminus V(C^-))$ and $E(G[V(C^-)])$, respectively. By above arguments, we have 
		\begin{equation*}
			m_1 \ge \Big\lfloor \frac{c}{2} \Big\rfloor \cdot (2(n-c)-(n-c))=\Big\lfloor \frac{c}{2} \Big\rfloor \cdot (n-c)\geq \frac{(c-1)(n-c)}{2}
		\end{equation*}
		and
		\begin{equation*}
			m_2 \ge \frac{c(c-1)}{2}-c.
		\end{equation*}
		Therefore, 
		\begin{align*}
			e(\dot{G})&\leq \frac{n(n-1)}{2}-(m_1+m_2)\\
			&\leq\frac{n(n-1)}{2}-\frac{(n-2)c-n}{2}\\
			&\leq\frac{n(n-1)}{2}-\frac{5n-12}{2}\\
			&<\frac{n(n-1)}{2}-2(n-3),
		\end{align*}
		which is a contradiction. 
		
		This completes the proof.
	\end{proof}
	
	Now we are in a position to give the proof of Theorem \ref{thm::1}.

	\renewcommand\proofname{\it Proof of Theorem \ref{thm::1}.}
	\begin{proof}
		Note that $C_{3}^{-} \cdot K_{n-2}$ contains no negative cycle of length $2k+1$ whenever $k\geq 3$ and $n\geq 10k+15$. Suppose that $\dot{G}=(G,\sigma)$  is an unbalanced $C_{2k+1}^-$-free signed graph of order $n$ with the maximum number of edges and $\dot{G}\nsim C_{3}^{-} \cdot K_{n-2}$. Clearly, $\dot{G}$ is connected, and
		\begin{equation}\label{eq::2}
			e(\dot{G}) \ge e(C_{3}^{-} \cdot K_{n-2}) = \frac{n(n-1)}{2}-2(n-3).
		\end{equation}
		Since $\dot{G}$ is unbalanced, there exists  some negative cycle in $\dot{G}$, and we take $C^{-}$ as one negative cycle of the minimum length.  By Lemma \ref{lem::3}, $C^{-}$ is of length at most $4$.

		Let $e=\{ uv \}$ be a negative edge in $C^-$. For any $x\in V(G)$, we define the distance between $x$ and  $e$, denoted by $d(x,e)$, as the minimum value of the distances between $x$ and one of the endpoints of $e$. Let $D:=\max\{d(x,e):x\in V(G)\}$, and let  $\Gamma_{i}$ ($0\leq i\leq D$) denote the set of vertices in $G$ at distance $i$ from $e$. In particular, $\Gamma_{0}=\{ u,v \}$. Note that there are no edges between $\Gamma_{i}$ and $\Gamma_{j}$ whenever $|i-j|>1$. Let $m$ denote the number of edges missing in $\dot{G}$. By \eqref{eq::2}, we have
		$$
		2(n-3)\geq m\geq |\Gamma_0|(n-|\Gamma_0|-|\Gamma_{1}|)+|\Gamma_{1}|(n-|\Gamma_0|-|\Gamma_{1}|-|\Gamma_{2}|),
		$$
		which implies that
		\begin{equation}\label{eq::1}
			|\Gamma_0|+|\Gamma_{1}|+|\Gamma_{2}| \ge n-1.
		\end{equation}
		Thus $D\leq 3$, and $|\Gamma_3|\leq 1$. We consider the following two cases. 
		
		{\flushleft {\it Case 1.} $|\Gamma_{1}|=1$.}
		
		In this case, $\sum_{i\geq 2}|\Gamma_{i}| =n-|\Gamma_{0}|-|\Gamma_{1}|= n-2-|\Gamma_{1}|=n-3$. Since there are no edges between $\{u,v\}$ and $\cup_{i\geq 2}\Gamma_{i}$, we have
		$$
		e(\dot{G}) \le \frac{n(n-1)}{2}-2(n-3).
		$$
		Combining this with \eqref{eq::2} yields that $G \cong C_{3} \cdot K_{n-2}$. By Lemma \ref{lem::2}, we assert that the subgraph $\dot{K}_{n-2}$ of $\dot{G}$ is switching equivalent to a balanced complete graph, since otherwise $\dot{K}_{n-2}$ would contain a $C_{2k+1}^{-}$. Note that the switching operation does not change the sign of cycles. Thus we conclude that  $\dot{G}\sim C_{3}^{-} \cdot K_{n-2}$, contrary to the assumption. 
		
		{\flushleft {\it Case 2.} $|\Gamma_{1}| \ge 2$.}
		
		We divide $\Gamma_{1}$ into three parts, namely $V_{1}$, $V_{2}$, and $V_{3}$, where $V_{1}=N(u)\setminus((N(u) \cap N(v))\cup\{v\})$, $V_{2}=N(u) \cap N(v)$, and $V_{3}=N(v)\setminus((N(u) \cap N(v))\cup\{u\})$. Let $\omega$ denote the clique number of $\dot{G}$. Note that 
		\begin{equation}
			\begin{aligned}
				e(\dot{G}) &\ge \frac{n(n-1)}{2}-2(n-3)=\frac{n^2-5n+12}{2}>\Big(1-\frac{1}{n/5}\Big)\frac{n^2}{2}.
				\nonumber
			\end{aligned}
		\end{equation}
		Then Lemma \ref{lem::1'} implies that  $\omega \geq n/5+1$. We have the following claims.

		\begin{claim}\label{claim::0} 
			There is a maximum clique  $K_{\omega}$ of  $\dot{G}$ such that $V(K_{\omega}) \cap \Gamma_{1} \neq \emptyset$.
		\end{claim}
		
		\renewcommand\proofname{\it Proof of Claim \ref{claim::0}.}
		\begin{proof}  
			Suppose to the contrary that for every maximum clique  $K_{\omega}$ of  $\dot{G}$, we have $V(K_{\omega}) \cap \Gamma_{1} = \emptyset$. In this situation, we also can deduce that $V(K_{\omega}) \cap \Gamma_{0} = \emptyset$ because $\omega\geq n/5+1>2$.  For any $x \in \Gamma_{1}$, we assert that $|E(\{ x \}, V(K_{\omega}))| \le \omega-1$, since otherwise $G[\{x \} \cup V(K_{\omega})]$ is a clique of order $\omega+1$ in $\dot{G}$, contrary to the maximality of $\omega$. If $|E(\{ x \}, V(K_{\omega}))|= \omega-1$ for some $x\in \Gamma_1$, then there is a clique of order $\omega$ containing  $x$, contrary to our assumption. Thus we may assume that  $|E(\{ x \}, V(K_{\omega}))| \le \omega-2$ for all $x\in \Gamma_1$. Then the number of edges missing in $\dot{G}$ satisfies the following inequality:
			\begin{equation}
				m \ge |\overline{E}(v, V_1)|+|\overline{E}(u, V_3)|+|\overline{E}(\Gamma_0, \Gamma_2)|+|\overline{E}(\Gamma_1, V(K_{\omega}) \cap \Gamma_2)|+|\overline{E}(\Gamma_0 \cup \Gamma_1, \Gamma_3)|.
				\nonumber
			\end{equation}
			If $\Gamma_3 = \emptyset$, then $V(K_\omega)\subseteq\Gamma_2$, and by  the above inequality, 
			\begin{equation*}
				m  \ge |V_1|+|V_3|+2|\Gamma_{2}|+2|\Gamma_{1}|=2n-4+|V_1|+|V_3|>2n-6.
			\end{equation*}
			If $\Gamma_3 \neq \emptyset$, then $|\Gamma_3|=1$, and by the above inequality,  
			\begin{equation*}
				m \ge |V_1|+|V_3|+2|\Gamma_{2}|+|\Gamma_{1}|+|\Gamma_0 \cup \Gamma_1|
				=2n-4+|V_1|+|V_3|
				>2n-6.
			\end{equation*}
			Therefore, we  always obtain a contradiction because  $\dot{G}$ only miss at most $2n-6$ edges by \eqref{eq::2}. 
		\end{proof}
		
		\begin{claim}\label{claim::1}
			There is a path $P$ of length at most $4$ containing $\{uv \}$ with both endpoints in $K_{\omega}$, where $K_\omega$ is the maximum clique given in Claim \ref{claim::0}.
		\end{claim}

		\renewcommand\proofname{\it Proof of Claim \ref{claim::1}.}
		\begin{proof} 
			
			If $K_{\omega}$ contains two vertices from different parts of $\Gamma_{1}$, then the result follows immediately. By Claim \ref{claim::0} and by the symmetry of $V_1$ and $V_3$, we only need to consider the following two situations.

			{\flushleft \it Case A.} $V(K_{\omega}) \cap \Gamma_{1} \subseteq V_{1}$. 
			
			Obviously, $V_{2} \cup V_{3}\neq \emptyset$ because $\{uv\}$ is an edge of the negative cycle $C^-$. We consider the following two cases. 
			
			{\flushleft \it Subcase A.1.}  $V(K_{\omega}) \subseteq \Gamma_{0} \cup \Gamma_{1}$.
			
			In this situation, we assert that $u\in V(K_{\omega})$, since otherwise $G[\{u \} \cup V(K_{\omega})]$ is a clique of order $\omega+1$ in $G$, contrary to the maximality of $\omega$. Also,  $v\not\in V(K_\omega)$ because $V(K_\omega)\cap \Gamma_1\subseteq V_1$. If $E(V_2 \cup V_3, V_1 \cap V(K_{\omega})) \neq \emptyset$, then there is clearly a path $P$ of length $3$ containing $\{uv\}$ with both endpoints in $K_\omega$. Thus we may assume that  $E(V_2 \cup V_3, V_1 \cap V(K_{\omega})) = \emptyset$. 
			
			If $V_1\setminus (V_1\cap V(K_{\omega})) = \emptyset$, i.e., $V(K_\omega)=\{u\}\cup V_1$, then the number of edges missing in $\dot{G}$ satisfies the following inequality:
			\begin{equation}
				\begin{aligned}
					m & \ge |\overline{E}(V_2 \cup V_3, V_1 \cap V(K_{\omega}))|+|\overline{E}(\{v\}, V_1)|+|\overline{E}(\{u\}, V_3)|+|\overline{E}(\Gamma_0, \Gamma_2)|\\
					&~~~~+|\overline{E}(\Gamma_0 \cup \Gamma_1, \Gamma_3)|\\
					& = (|V_2|+|V_3|)(\omega-1)+|V_1|+|V_3|+2|\Gamma_{2}|+|\Gamma_0 \cup \Gamma_1||\Gamma_3|\\
					&= (|V_2|+|V_3|)(\omega-1)+|V_1|+|V_3|+2(|\Gamma_{2}|+|\Gamma_{3}|)+(|\Gamma_0 \cup \Gamma_1|-2)|\Gamma_3|\\
					&\ge (|V_2|+|V_3|)(\omega-1)+|V_1|+|V_3|+2(|\Gamma_{2}|+|\Gamma_{3}|)\\
					&= (|V_2|+|V_3|)(\omega-1)+|V_1|+|V_3|+2(n-2-|V_1|-|V_2|-|V_3|)\\
					& =2n-4+(\omega-3)|V_2|+(\omega-2)|V_3|-|V_{1}|\\
					& \ge 2n-4+(\omega-3)(|V_2|+|V_3|)-|V_1|\\
					& \ge 2n-4+(\omega-3)-|V_1|\\
					& = 2n-6.
					\nonumber
				\end{aligned}
			\end{equation}
			Combining this with  (\ref{eq::2}), we obtain  $m=2n-6$, and hence $|V_2|=1$, $|V_3|=0$, $|\Gamma_3| =0$ and $G$ only misses the edges in $\overline{E}(V_2 \cup V_3, V_1 \cap V(K_{\omega}))\cup \overline{E}(\{v\}, V_1)\cup \overline{E}(\{u\}, V_3)\cup \overline{E}(\Gamma_0, \Gamma_2)$. If $\Gamma_2 \neq \emptyset$, then there is clearly a path $P$ of length $4$ containing $\{uv\}$ with  both endpoints in $K_\omega$. If $\Gamma_2 = \emptyset$, then $G$ is isomorphic to $C_{3} \cdot K_{n-2}$. Since $\{uv\}$ is a negative edge of  $C^-$, we see that $\dot{G}[V(C_3)]$ is just the negative cycle $C^-$. If $\dot{G}[V(K_{n-2})]$ is unbalanced, by Lemma \ref{lem::2}, $\dot{G}[V(K_{n-2})]$ contains $C_{2k+1}^-$, which is impossible. Thus $\dot{G}[V(K_{n-2})]$ is balanced, and $\dot{G} \sim C_{3}^{-} \cdot K_{n-2}$, contrary to our assumption. 
			
			If $V_1 \setminus (V_1 \cap V(K_{\omega})) \neq \emptyset$, then each vertex of $V_1 \setminus (V_1 \cap V(K_{\omega}))$ is not adjacent to at least one vertex in $V_1 \cap V(K_{\omega})$. Thus we have
			\begin{equation}
				\begin{aligned}
					m & \ge |\overline{E}(V_2 \cup V_3, V_1 \cap V(K_{\omega}))|+|V_1\setminus(V_1 \cap V(K_{\omega}))|+|\overline{E}(\{v\}, V_1)|\\
					&~~~~+|\overline{E}(\{u\}, V_3)|+|\overline{E}(\Gamma_0, \Gamma_2)|+|\overline{E}(\Gamma_0 \cup \Gamma_1, \Gamma_3)|\\
					& = (|V_2|+|V_3|)(\omega-1)+(|V_1|-\omega+1)+|V_1|+|V_3|+2|\Gamma_{2}|+|\Gamma_0 \cup \Gamma_1||\Gamma_3|\\
					& = (|V_2|+|V_3|)(\omega-1)+(|V_1|-\omega+1)+|V_1|+|V_3|+2(|\Gamma_{2}|+|\Gamma_3|)\\
					&~~~~+(|\Gamma_0 \cup \Gamma_1|-2)|\Gamma_3|\\
					& \ge (|V_2|+|V_3|)(\omega-1)+(|V_1|-\omega+1)+|V_1|+|V_3|+2(|\Gamma_{2}|+|\Gamma_3|)\\
					&\ge (|V_2|+|V_3|)(\omega-1)+(|V_1|-\omega+1)+|V_1|+|V_3|+2(n-2-|V_1|-|V_2|-|V_3|)\\
					& =2n-3+(\omega-3)|V_2|+(\omega-2)|V_3|-\omega\\
					& \ge 2n-3+(\omega-3)(|V_2|+|V_3|)-\omega\\
					& \ge 2n-3+(\omega-3)-\omega\\
					&= 2n-6.
					\nonumber
				\end{aligned}
			\end{equation}
			Combining this with  (\ref{eq::2}), we obtain $m=2n-6$, and hence $|V_2|=1$, $|V_3|=0$, $|\Gamma_3| =0$ and $G$ only misses the edges counting in the right hand side of the first inequality. Therefore, $E(V_1\setminus(V_1 \cap V(K_{\omega}), V_2 \cup V_3)) \neq \emptyset$, and we can easily find  a path $P$ of length $4$ containing $\{uv\}$ with both endpoints in $K_\omega$.
			
			{\flushleft {\it Subcase A.2.} $V(K_{\omega}) \subseteq \Gamma_{1} \cup \Gamma_{2}$.}
			
			In this situation, we assert that $V(K_{\omega}) \cap \Gamma_{2} \neq \emptyset$, since otherwise $G[\{u \} \cup V(K_{\omega})]$ is a clique of order $\omega+1$ in $G$, contrary to the maximality of $\omega$. Let $x \in V_{1} \cap V(K_{\omega})$. If $E(V_2 \cup V_3, V(K_{\omega}) \setminus \{x \} ) \neq \emptyset$, then there is clearly a path $P$ of length $4$ containing $\{uv\}$ whose two endpoints are in $K_\omega$. Thus we may assume that  $E(V_2 \cup V_3, V(K_{\omega}) \setminus \{x \} ) = \emptyset$. 
			
			If $(V_1 \cup \Gamma_{2}) \setminus V(K_{\omega}) = \emptyset$, i.e., $V(K_\omega)=V_1 \cup \Gamma_{2}$, then the number of edges missing in $\dot{G}$ satisfies the following inequality:
			\begin{equation}
				\begin{aligned}
					m & \ge |\overline{E}(V_2 \cup V_3, V(K_{\omega})\setminus \{x \})|+|\overline{E}(v, V_1)|+|\overline{E}(u, V_3)|+|\overline{E}(\Gamma_0, \Gamma_2)|\\
					&~~~~+|\overline{E}(\Gamma_0 \cup \Gamma_1, \Gamma_3)|\\
					& = (|V_2|+|V_3|)(\omega-1)+|V_1|+|V_3|+2|\Gamma_{2}|+|\Gamma_0 \cup \Gamma_1||\Gamma_3|\\
					& = (|V_2|+|V_3|)(\omega-1)+|V_1|+|V_3|+2(|\Gamma_{2}|+|\Gamma_{3}|)+(|\Gamma_0 \cup \Gamma_1|-2)|\Gamma_3|\\
					&\ge  (|V_2|+|V_3|)(\omega-1)+|V_1|+|V_3|+2(|\Gamma_{2}|+|\Gamma_{3}|)\\
					&=(|V_2|+|V_3|)(\omega-1)+|V_1|+|V_3|+2(n-2-|V_1|-|V_2|-|V_3|)\\
					& =2n-4+(\omega-3)(|V_2|)+(\omega-2)|V_3|-|V_1|\\
					& \ge 2n-4+(\omega-3)(|V_2|+|V_3|)-|V_1|\\
					& \ge 2n-4+(\omega-3)-|V_1|\\
					& = 2n-6+(\omega-|V_1|-1)\\
					&\ge 2n-6.
					\nonumber
				\end{aligned}
			\end{equation}
			Combining this with  (\ref{eq::2}), we obtain $m=2n-6$ , and hence $|V_1|=\omega-1$, $|V_2|=1$, $|V_3|=0$, $|\Gamma_3| =0$ and $\dot{G}$ only misses the edges in $\overline{E}(V_2 \cup V_3, V(K_{\omega}) \setminus \{x \}) \cup \overline{E}(\{v\}, V_1)\cup \overline{E}(\{u\}, V_3)\cup \overline{E}(\Gamma_0, \Gamma_2)$. It follows that $E(\{ x \}, V_2 \cup V_3) \neq \emptyset$ and $|V_1 \cap V(K_{\omega})|=\omega-1 \ge n/5 > 1$. Therefore, we can easily find a path of length $4$ containing $\{uv\}$ whose two endpoints are in $K_\omega$.
			
			If $(V_1 \cup \Gamma_{2}) \setminus V(K_{\omega}) \neq \emptyset$, then each vertex in $(V_1 \cup \Gamma_{2}) \setminus V(K_{\omega})$ is not adjacent to at least one vertex in $K_{\omega}$.  Note that $n \ge |V_1|+4$. We have
			\begin{equation}
				\begin{aligned}
					m &\ge |\overline{E}(V_2 \cup V_3, V(K_{\omega})\setminus \{x \})|+|(V_1 \cup \Gamma_{2}) \setminus V(K_{\omega})|+|\overline{E}(v, V_1)|+|\overline{E}(u, V_3)|\\
					&~~~~+|\overline{E}(\Gamma_0, \Gamma_2)|+|\overline{E}(\Gamma_0 \cup \Gamma_1, \Gamma_3)|\\ 
					&= (|V_2|+|V_3|)(\omega-1)+(|V_1|+|\Gamma_{2}|-\omega)+|V_1|+|V_3|+2|\Gamma_{2}|+|\Gamma_0 \cup \Gamma_1||\Gamma_3|\\
					&= (|V_2|+|V_3|)(\omega-1)+2|V_1|+|V_3|-\omega+3(|\Gamma_{2}|+|\Gamma_3|)+(|\Gamma_0 \cup \Gamma_1|-3)|\Gamma_3|\\
					&\ge (|V_2|+|V_3|)(\omega-1)+2|V_1|+|V_3|-\omega+3(|\Gamma_{2}|+|\Gamma_3|)\\
					&\ge (|V_2|+|V_3|)(\omega-1)+2|V_1|+|V_3|-\omega+3(n-2-|V_1|-|V_2|-|V_3|)\\
					&=3n-6-|V_1|+(\omega-4)|V_2|+(\omega-3)|V_3|-\omega\\
					&\ge 3n-6-|V_1|+(\omega-4)(|V_2|+|V_3|)-\omega\\
					&\ge 3n-6-|V_1|+(\omega-4)-\omega\\
					&= 2n-6+(n-|V_1|-4)\\
					&\ge 2n-6.
					\nonumber
				\end{aligned}
			\end{equation}
			Combining this with  (\ref{eq::2}), we obtain $m=2n-6$, and hence $|V_1|=n-4$, $|V_2|=1$, $|V_3|=0$, $|\Gamma_3| =0$ and $\dot{G}$ only misses the edges counting in the right hand side of the first inequality. It follows that $E(\{x \}, V_2 \cup V_3) \neq \emptyset$ and $|V_1 \cap V(K_{\omega})|=\omega-1 \ge n/5 >1$. Therefore, we can easily find a path $P$ of length $4$ containing $\{uv\}$ with both endpoints in $K_\omega$.

			{\flushleft {\it Case B.} $V(K_{\omega}) \cap \Gamma_{1} \subseteq V_{2}$. }
			
			If $|V_{2} \cap V(K_{\omega})| \ge 2$, then the result follows immediately. We may assume that $|V_{2} \cap V(K_{\omega})| = 1$. Let $V_{2} \cap V(K_{\omega})=\{ x \}$. If $E(V_1 \cup (V_2 \setminus \{x \}) \cup V_3, V(K_{\omega}) \setminus \{x \}) \neq \emptyset$, then there is clearly a path of length at most $4$ containing $\{uv\}$ with both endpoints in $K_\omega$. Thus we assume that $E(V_1 \cup (V_2 \setminus \{x \}) \cup V_3, V(K_{\omega}) \setminus \{x \}) = \emptyset$, and the number of edges missing in $\dot{G}$ satisfies the following inequality:

			\begin{figure}
				\begin{center}
					\includegraphics[width=14cm,height=3.4cm]{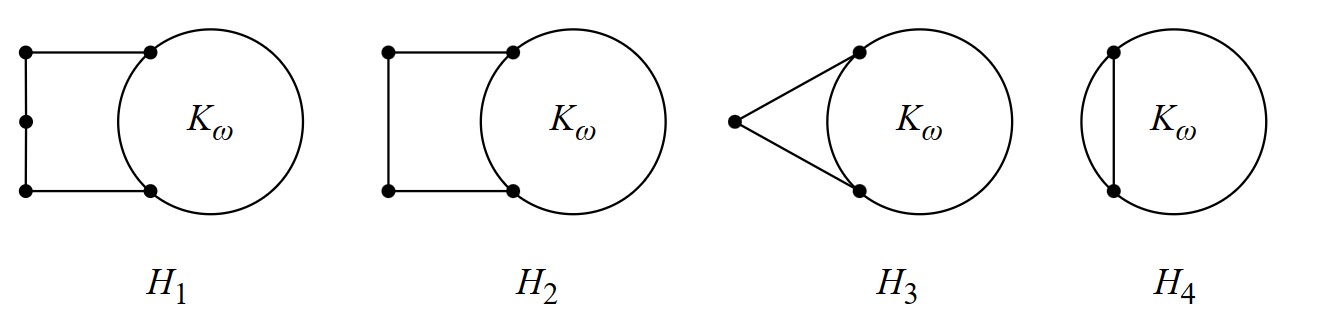}
				\end{center}
				\vskip -0.5cm
				\caption{All possible structures for $P \cup K_{\omega}$.}
				\label{fig4}
			\end{figure}

			\begin{equation}
				\begin{aligned}
					m & \ge |\overline{E}(V_1 \cup (V_2 \setminus \{x \}) \cup V_3, V(K_{\omega}) \setminus \{x \})|+|\overline{E}(v, V_1)|+|\overline{E}(u, V_3)|\\
					&~~~~+|\overline{E}(\Gamma_0, \Gamma_2)|+|\overline{E}(\Gamma_0 \cup \Gamma_1, \Gamma_3)|\\
					& = (|V_1|+|V_2|-1+|V_3|)(\omega-1)+|V_1|+|V_3|+2|\Gamma_{2}|+|\Gamma_0 \cup \Gamma_1||\Gamma_3|\\
					& = (|V_1|+|V_2|-1+|V_3|)(\omega-1)+|V_1|+|V_3|+2(|\Gamma_{2}|+|\Gamma_{3}|)+(|\Gamma_0 \cup \Gamma_1|-2)|\Gamma_3|\\
					& \ge (|V_1|+|V_2|-1+|V_3|)(\omega-1)+|V_1|+|V_3|+2(|\Gamma_{2}|+|\Gamma_{3}|)\\
					&= (|V_1|+|V_2|-1+|V_3|)(\omega-1)+|V_1|+|V_3|+2(n-2-|V_1|-|V_2|-|V_3|)\\
					& =2n-3+(\omega-2)(|V_1|+|V_3|)+(\omega-3)|V_2|-\omega \\
					& \ge 2n-6+(\omega-2)(|V_1|+|V_3|)\\
					&\ge 2n-6.
					\nonumber
				\end{aligned}
			\end{equation}
			Combining this with  (\ref{eq::2}), we obtain $m=2n-6$, and hence $|V_1|=|V_3|=0$, $|V_2|=1$, $|\Gamma_3| =0$ and $G$ is isomorphic to $C_{3} \cdot K_{n-2}$. By using a similar method as in Subcase A.1, we can deduce that $\dot{G} \sim C_{3}^{-} \cdot K_{n-2}$, contrary to our assumption. 
			
			This completes the proof of  Claim \ref{claim::1}.
		\end{proof}
		
		According to Claim \ref{claim::1}, $P \cup K_{\omega}$ is isomorphic to one of the graphs shown in Figure \ref{fig4}, that is, $P \cup K_{\omega}\cong H_{|P \cup K_{\omega}|,a}$ for some $2 \le a \le 5$.  If $\dot{P} \cup \dot{K}_{\omega}$ is unbalanced, by Lemmas \ref{lem::2} and  \ref{lem::2'},  we assert that  $\dot{G}$ contains a negative cycle of length  $2k+1$  because $5 \le 2k+1 \le n/5-1\le \omega-2$, contrary to our assumption. If $\dot{P} \cup \dot{K}_{\omega}$ is balanced, up to switching equivalence, we can assume that $\sigma(e)=+1$ for all $e \in E(P \cup K_{\omega})$. 
		Considering that $5 \le 2k+1 \le n/5-1\le \omega-2$ and $\{uv\}\in E(P)$ is also contained in the negative cycle $C^-$ of length at most $4$, we can easily obtain a negative cycle of length  $2k+1$ in $C^-\cup \dot{P} \cup \dot{K}_{\omega}$, which is a contradiction.
		
		We complete the proof.
	\end{proof}

	\section{Proof of Theorem \ref{thm::2}}
	In this section, we shall give the proof of Theorem \ref{thm::2} with the help of Theorem \ref{thm::1}. Before going further, we need some lemmas.
	
	\begin{lem}(Hong \cite{H})\label{lem::7}
		Let $G$ be a connected graph on $n$ vertices. Then
		$$\rho(G)\le \sqrt{2e(G)-n+1}.$$
	\end{lem}
	
	The \emph{frustration index} of a signed graph $\dot{G}$ is the minimum number of edges in $\dot{G}$ whose removal results in a balanced signed graph. 
	
	\begin{lem}(Stani\'c \cite{ZS19})\label{lem::5}
		Let $\dot{G}$ be a connected signed graph on $n$ vertices with  frustration index $l$. Then
		\begin{equation}
			\lambda_1(\dot{G}) \le \sqrt{2(e(\dot{G})-l)-n+1}.
			\nonumber
		\end{equation}
	\end{lem}
	
	$$\lambda_1(G)\leqslant\left(1-\frac{1}{\omega(G)}\right)n.$$
	
	
	The \textit{balanced clique number} of a signed graph $\dot{G}$, denoted by $\omega_b(\dot{G})$, is the maximum order of a  balanced clique in $\dot{G}$.
	
	\begin{lem}(Wang, Yan and Qian \cite{WYQ})\label{WYQ}
		Let $\dot{G}$ be a signed graph of order $n$. Then
		$$\lambda_1(\dot{G})\leqslant\left(1-\frac{1}{\omega_b(\dot{G})}\right)n.$$
	\end{lem}

	\begin{lem}\label{lem::8}
		Let $n \ge 4$ be integer. If $\dot{G} \sim C_{3}^- \cdot K_{n-2}$, then the characteristic polynomial of $\dot{G}$ is 
		\begin{equation}
			P_{\dot{G}}(x)=(x+1)^{n-4}(x-1)(x^3-(n-5)x^2-(2n-5)x+n-5).
			\nonumber
		\end{equation}
		In particular,
		\begin{equation}
			\rho(\dot{G}) > n-3.
			\nonumber
		\end{equation}
	\end{lem}
	
	\renewcommand\proofname{\it Proof.}
	
	\begin{proof}
		Let $V(C_{3}^-)=\{ v_1, v_2, v_3 \}$ and $V(C_{3}^-) \cap V(K_{n-2}) = \{ v_3 \}$. Up to switching equivalence, we can assume that $\{v_1v_2\}$  is the unique negative edge in  $\dot{G}$. Then  the adjacency matrix of $\dot{G}$ is given by
		$$
		M=\begin{pmatrix}
			I_2-J_2&\boldsymbol{j}_2&O\\
			\boldsymbol{j}_2^{T}&0&\boldsymbol{j}_{n-3}^{T}\\
			O&\boldsymbol{j}_{n-3}&J_{n-3}\\
		\end{pmatrix},
		$$
		where $I$, $J$ and $\boldsymbol{j}$ denote the identity matrix, the all-ones matrix, and the all-ones vector, respectively. Now the characteristic polynomial of $\dot{G}$ is
		$$
		\begin{aligned}
			P_{\dot{G}}(x)&=
			\left|\begin{matrix}
				(x-1)I_2+J_2&-\boldsymbol{j}_2&O\\
				-\boldsymbol{j}_2^{T}&x&-\boldsymbol{j}_{n-3}^{T}\\
				O&-\boldsymbol{j}_{n-3}&xI_{n-3}-J_{n-3}\\
			\end{matrix}\right|\\
			&=\left|\begin{matrix}
				(x-1)I_2+J_2&-\boldsymbol{j}_2&O\\
				-\boldsymbol{j}_2^{T}&x&-\boldsymbol{j}_{n-3}^{T}\\
				O&0&xI_{n-3}-J_{n-3}\\
			\end{matrix}\right|+
			\left|\begin{matrix}
				(x-1)I_2+J_2&0&O\\
				-\boldsymbol{j}_2^{T}&x&-\boldsymbol{j}_{n-3}^{T}\\
				O&-\boldsymbol{j}_{n-3}&xI_{n-3}-J_{n-3}\\
			\end{matrix}\right|\\
			&~~~~-
			\left|\begin{matrix}
				(x-1)I_2+J_2&0&O\\
				-\boldsymbol{j}_2^{T}&x&-\boldsymbol{j}_{n-3}^{T}\\
				O&0&xI_{n-3}-J_{n-3}\\
			\end{matrix}\right|\\
			&=(x+2)(x-1)^2(x-(n-4))(x+1)^{n-4}+(x-1)(x+1)(x-(n-3))(x+1)^{n-3}\\	&~~~~-x(x-1)(x+1)(x-(n-4))(x+1)^{n-4}\\
			&=(x+1)^{n-4}(x-1)(x^3-(n-5)x^2-(2n-5)x+n-5).
		\end{aligned}
		$$
		It is easy to see that $\lambda_1(\dot{G})$ is the zero of the polynomial $f(x):=x^3-(n-5)x^2-(2n-5)x+n-5$.  By a simple calculation, we see that $f(n-3)=-2<0$. Therefore, $\rho(\dot{G})\geq \lambda_1(\dot{G})>n-3$, and the result follows.  
	\end{proof}

	Now we are in a position to give the proof of Theorem \ref{thm::2}.
	
	\renewcommand\proofname{\it Proof of Theorem \ref{thm::2}.}
	\begin{proof}
		Suppose that $\dot{G}=(G,\sigma)$  is an unbalanced $C_{2k+1}^-$-free signed graph of order $n$ with the maximum spectral radius, where $3\leq k\leq (n-11)/10$. Then  
		\begin{equation}\label{eq::3}
			\rho(\dot{G}) \ge \rho(C_{3}^{-} \cdot K_{n-2})>n-3
		\end{equation}
		by Lemma \ref{lem::8}. We claim that  $\rho(\dot{G})=\lambda_1(\dot{G})$. If not, we have $\rho(\dot{G})=\max\{\lambda_1(\dot{G}),$ $-\lambda_n(\dot{G})\}=-\lambda_n(\dot{G})$.  Let $\dot{G}_{1}=-\dot{G}$. Clearly,  $\lambda_1(\dot{G}_{1})=-\lambda_n(\dot{G})$. Since $\dot{G}$ is $C_{2k+1}^-$-free,  $\dot{G}_{1}$ contains no positive cycles of length $2k+1$, and it   follows that $\omega_b(\dot{G}_{1}) \le 2k \le (n-11)/5$.  According to Lemma \ref{WYQ}, 
		$$
		\rho(\dot{G})=-\lambda_n(\dot{G})=\lambda_1(\dot{G_1}) \le \left(1-\frac{1}{\omega_b(\dot{G_1})}\right)n \le n-\frac{5}{1-\frac{11}{n}}<n-3,
		$$
		which is a contradiction. \qed
		
		First assume that $\dot{G}$ is connected. By Lemma \ref{lem::5}, we have
		\begin{equation*}
			\rho(\dot{G}) 
			\le \sqrt{2(e(\dot{G})-1)-n+1}
			= \sqrt{2e(\dot{G})-n-1}.
		\end{equation*}
		Combining this with \eqref{eq::3} yields that 
		\begin{equation*}
			e(\dot{G}) \ge \frac{\rho(\dot{G})^2+n+1}{2}> \frac{(n-3)^2+n+1}{2}= \frac{n(n-1)}{2}-2n+5.
		\end{equation*}
		Therefore, $e(\dot{G}) \ge \frac{n(n-1)}{2}-2(n-3)$. Since $\dot{G}$ is $C_{2k+1}^-$-free and $3\leq (n-11)/10<n/10-1$, by Theorem \ref{thm::1}, we conclude that  $\dot{G}$ is switching equivalent to $C_3^-\cdot K_{n-2}$, and the result follows.
		
		Now assume that  $\dot{G}$ is disconnected. Since $\rho(\dot{G})> n-3$, we see that $\dot{G}$ has exactly two components, and one of them consists of an isolated vertex. Let $\dot{G}_{2}$ be the nontrivial component of $\dot{G}$. Then $\dot{G}_{2}$ is of order $n-1$ and $\rho(\dot{G}_2)=\rho(\dot{G})>n-3$.  By Lemma \ref{lem::8}, 
		\begin{equation}
			\begin{aligned}
				e(\dot{G}_2) &\ge \frac{\rho(\dot{G}_2)^2+n+1}{2}> \frac{(n-3)^2+n+1}{2}> \frac{(n-1)(n-2)}{2}-2(n-1)+6.
				\nonumber
			\end{aligned}
		\end{equation}
		By Theorem \ref{thm::1}, $\dot{G}_2$ contains a $C_{2k+1}^{-}$ because $3 \le k \le (n-1)/10-1$. Thus $\dot{G}$ also contains $C_{2k+1}^{-}$, and we obtain a contradiction.
		
		This completes the proof.
	\end{proof}

\end{document}